\documentclass[10pt,leqno,letterpaper]{article}
\usepackage[utf8]{inputenc}
\usepackage{amsmath}
\usepackage{amsfonts}
\usepackage{amssymb}
\usepackage{graphicx}
\usepackage{mathrsfs}
\usepackage{upref,amsthm,amsxtra,exscale}
\usepackage{cite}
\usepackage[colorlinks=true,urlcolor=blue,
citecolor=red,linkcolor=blue,linktocpage,pdfpagelabels,
bookmarksnumbered,bookmarksopen]{hyperref}

\newtheorem{theorem}{Theorem}[section]

\newtheorem{remark}[theorem]{Remark}
\newtheorem{lemma}[theorem]{Lemma}

\newtheorem{definition}[theorem]{Definition}

\numberwithin{equation}{section}

\def\r{\mathbb{R}}
\def\rn{\mathbb{R}^N}
\def\z{\mathbb{Z}}

\def\n{\mathbb{N}}

\def\eps{\varepsilon}

\def\io{\int_{\Omega}}

\def\vp{\varphi}
\def\o{\Omega}
\def\bf{\mathbf}
\def\tilde{\widetilde}
\def\cC{\mathcal{C}}

\def\cH{\mathcal{H}}
\def\cI{\mathcal{I}}
\def\cJ{\mathcal{J}}
\def\cK{\mathcal{K}}
\def\cM{\mathcal{M}}
\def\cN{\mathcal{N}}

\def\cS{\mathcal{S}}
\def\cT{\mathcal{T}}
\def\cU{\mathcal{U}}

\author{Mónica Clapp\footnote{M. Clapp was partially supported by CONACYT grant A1-S-10457 (Mexico)}\qquad and \qquad Angela Pistoia\footnote{A. Pistoia was partially supported by Fondi di Ateneo ``Sapienza'' Università di Roma (Italy).}}
\title{Fully nontrivial solutions to elliptic systems with mixed couplings}
\date{\today}

\begin{document}
\maketitle

\begin{abstract}
We study the existence of fully nontrivial solutions to the system
$$-\Delta u_i+ \lambda_iu_i = \sum\limits_{j=1}^\ell \beta_{ij}|u_j|^p|u_i|^{p-2}u_i\ \hbox{in}\ \Omega, \qquad
 i=1,\ldots,\ell,$$
in a bounded or unbounded domain $\Omega$ in $\mathbb R^N,$ $N\ge 3$. The $\lambda_i$'s are real numbers, and the nonlinear term may have subcritical ($1<p<\frac{N}{N-2}$), critical ($p=\frac{N}{N-2}$), or supercritical growth ($p>\frac{N}{N-2}$). The matrix $(\beta_{ij})$ is symmetric and admits a block decomposition such that the diagonal entries $\beta_{ii}$ are positive, the interaction forces within each block are attractive (i.e., all entries $\beta_{ij}$ in each block are non-negative) and the interaction forces between different blocks are repulsive (i.e., all other entries are non-positive). We obtain new existence and multiplicity results of fully nontrivial solutions, i.e., solutions where every component $u_i$ is nontrivial. We also find fully synchronized solutions (i.e., $u_i=c_i u_1$ for all $i=2,\ldots,\ell$)  in the purely cooperative case whenever $p\in(1,2).$ 


\end{abstract}

\section{Introduction}

We consider the system of nonlinear elliptic equations
\begin{equation} \label{eq:system}
\begin{cases}
-\Delta u_i+ \lambda_iu_i = \sum\limits_{j=1}^\ell \beta_{ij}|u_j|^p|u_i|^{p-2}u_i, \\
u_i\in H,\qquad i=1,\ldots,\ell,
\end{cases}
\end{equation}
where $H$ is either $H_0^1(\o)$ or $D^{1,2}_0(\o)$, $\o$ is an open subset of $\rn$, $N\geq 3$, $\lambda_i\in\r$ and $p>1$. We assume that
\begin{itemize}
\item[$(A_1)$] the operators $-\Delta+\lambda_i$ are well defined and coercive in $H$ for all $i=1,\ldots,\ell$,
\item[$(B_1)$] the matrix $(\beta_{ij})$ is symmetric and admits a block decomposition as follows: For some $1 < q < \ell$ there exist \ $0=\ell_0<\ell_1<\dots<\ell_{q-1}<\ell_q=\ell$ \ such that, if we set
\begin{align*}
 &I_h:= \{i \in  \{1,\dots,\ell\}:  \ell_{h-1} < i \le \ell_h \},\\
 &\cI_h:=I_h\times I_h,\qquad \cK_h:=\big\{(i,j)\in I_h\times I_k: k\in\{1,\ldots,q\}\smallsetminus\{h\}\big\},
\end{align*}
then \ $\beta_{ii}>0$,
\begin{equation*}
\beta_{ij}\ge0\ \text{ if }\ (i,j)\in \cI_h \quad \text{ and }\quad \beta_{ij}\le0\ \text{if}\ (i,j)\in \cK_h,\quad h=1,\ldots,q.
\end{equation*}
\end{itemize}

This type of systems models some physical phenomena in nonlinear optics and describes the behavior of multi-component Bose-Einstein condensates. The coefficient $\beta_{ij}$ represents the interaction force between the components $u_i$ and $u_j$. The sign of $\beta_{ij}$ determines whether the interaction is attractive or repulsive. If $\beta_{ij}\geq 0$ for all $i\neq j$ (i.e., if $q=1$) the system \eqref{eq:system} is called \emph{purely cooperative}, and it is called \emph{purely competitive} if $\beta_{ij}\leq 0$ for all $i\neq j$ (i.e., if $q=\ell$).

In the past fifteen years, systems that are either purely cooperative or purely competitive have been extensively studied, particularly those with cubic nonlinearity (i.e., with $p=2$). It is convenient to consider other powers, specially when dealing with critical systems; see, e.g., \cite{cz1,cz2,cp,cs}. We refer to the introduction of the papers \cite{bks,dp} for an overview on the topic and an ample list of references.

Systems with mixed couplings were considered in the seminal paper \cite{lw} by Lin and Wei and more recently in \cite{bsw,so,st,sw1,sw2,ty,ww}. All of these works treat only the cubic nonlinearity $p=2$. In the present paper, we are mainly concerned with the case $p<2$.

According to the decomposition given by $(B_1)$, we shall write a solution $\bf u=(u_1,\ldots,u_\ell)$ to \eqref{eq:system} in block-form as
$$\bf u=(\bar u_1,\ldots,\bar u_q)\qquad\text{with \ }\bar u_h=(u_{\ell_{h-1}+1},\ldots,u_{\ell_h}).$$
$\bf u$ is called \emph{semitrivial} if some but not all of its components $u_i$ are zero and it is said to be \emph{fully nontrivial} if every component $u_i$ is different from zero. We shall call it \emph{block-wise nontrivial} if at least one component of each block $\bar u_h$ is nontrivial.

We prove the following result.

\begin{theorem} \label{thm:main_block}
Assume $(A_1)$ and $(B_1)$. Assume further that
\begin{itemize}
\item[$(A_2)$] the embedding $H\hookrightarrow L^{2p}(\o)$ is compact.
\end{itemize}
Then, the system \eqref{eq:system} has a least energy block-wise nontrivial solution.
\end{theorem}

The precise meaning of \emph{least energy block-wise nontrivial solution} is given in Definition \ref{def:least energy}. The proof of this result is obtained by adapting the variational approach introduced in \cite{cs}, and is given in Section \ref{sec:variational}.

If the system is purely competitive (i.e., $q=\ell$) any block-wise nontrivial solution is fully nontrivial. On the other hand, for any choice of $i_h\in I_h$, every solution to the purely competitive system
\begin{equation*} 
\begin{cases}
-\Delta v_h+ \lambda_{i_h}v_h = \sum\limits_{k=1}^\ell \beta_{i_hi_k}|v_k|^p|v_h|^{p-2}v_h, \\
v_h\in H,\qquad h=1,\ldots,q,
\end{cases}
\end{equation*}
gives rise to a block-wise nontrivial solution of \eqref{eq:system} whose $i_h$-th component is $v_h$, $h=1,\ldots,q$, and all other components are $0$. If $q\neq\ell$ this solution is not fully nontrivial. The following result, whose proof is given in Section \ref{sec:fully nontrivial}, provides existence of a fully nontrivial solution.

\begin{theorem} \label{thm:main_fullynontrivial}
Assume $(A_1)$ and $(B_1)$, and let $p\leq 2$. There exists a positive constant $C_*$ independent of $(\beta_{ij})$ - but depending on $\o$, $\lambda_i$, $p$ and $q$ - with the property that, if
\begin{itemize}
\item[$(B_2)$] either $p<2$ and 
$$\min_{\substack{i,j\in I_h \\ i\neq j}}\beta_{ij}\left(\frac{\min\limits_{k=1,\ldots,q}\max\limits_{i\in I_k}\beta_{ii}}{\max\limits_{i,j\in I_h}\beta_{ij}}\right)^\frac{p}{p-1}>C_*(\ell_h-\ell_{h-1}-1)^\frac{2p}{p-1}\sum_{(i,j)\in\cK_h}|\beta_{ij}|,$$
for every $h=1,\ldots,q$,
\item[] or $p=2$, \ $\lambda_i=:a_h$ for all $i\in I_h$, \ $\beta_{ij}=:b_h$ for all $i,j\in I_h$ with $i\neq j$ and 
$$b_h>\max_{i\in I_h}\beta_{ii}+C_*\left[\frac{(\ell_h-\ell_{h-1}-1)\,b_h}{\min\limits_{k=1,\ldots,q}\max\limits_{i\in I_k}\beta_{ii}}\right]^2\max_{i,j\in I_h}\sum_{\substack{m\in I_k \\ k\neq h}}|\beta_{im}-\beta_{jm}|$$
for every $h=1,\ldots,q$,
\end{itemize}
then every least energy block-wise nontrivial solution to the system \eqref{eq:system} is fully nontrivial.
\end{theorem}

If $\o$ is a bounded domain, assumption $(A_1)$ holds true if $\lambda_i>-\lambda_1(\o)$, where $\lambda_1(\o)$ is the first Dirichlet eigenvalue of $-\Delta$ in $\o$, and $(A_2)$ is satisfied if the nonlinear term is subcritical. So combining Theorems \ref{thm:main_block} and \ref{thm:main_fullynontrivial} we obtain the following result.

\begin{theorem} \label{thm:bounded}
If $\o$ is bounded, $\lambda_i>-\lambda_1(\o)$ for all $i=1,\ldots,\ell$, $1<p<\frac{N}{N-2}$ and $(\beta_{ij})$ satisfies $(B_1)$ and $(B_2)$, the system
\begin{equation*}
\begin{cases}
-\Delta u_i+ \lambda_iu_i = \sum\limits_{j=1}^\ell \beta_{ij}|u_j|^p|u_i|^{p-2}u_i, \\
u_i\in H_0^1(\o),\qquad i=1,\ldots,\ell,
\end{cases}
\end{equation*}
has a fully nontrivial solution.
\end{theorem}

It is well known that compactness is more likely to hold true in a symmetric setting. Symmetries are also helpful to obtain sign-changing solutions. Symmetric versions of Theorems \ref{thm:main_block} and \ref{thm:main_fullynontrivial} yield the following results. 

\begin{theorem} \label{thm:entire subcritical}
If $1<p<\frac{N}{N-2}$ and $(\beta_{ij})$ satisfies $(B_1)$ and $(B_2)$, the system
\begin{equation*} 
\begin{cases}
-\Delta u_i+ u_i = \sum\limits_{j=1}^\ell \beta_{ij}|u_j|^p|u_i|^{p-2}u_i, \\
u_i\in H^1(\rn),\qquad i=1,\ldots,\ell,
\end{cases}
\end{equation*}
has a fully nontrivial solution whose components are positive and radial. 

If $N=4$ or $N\geq 6$ it has also and a fully nontrivial solution whose components are nonradial and change sign.
\end{theorem}

\begin{theorem} \label{thm:entire critical}
If $p=\frac{N}{N-2}$ and $(\beta_{ij})$ satisfies $(B_1)$ and $(B_2)$, the critical system
\begin{equation*}
\begin{cases}
-\Delta u_i = \sum\limits_{j=1}^\ell \beta_{ij}|u_j|^p|u_i|^{p-2}u_i, \\
u_i\in D^{1,2}(\rn),\qquad i=1,\ldots,\ell,
\end{cases}
\end{equation*}
has a fully nontrivial solution whose components are positive.

If $N=3$ or $N\geq 5$, it has also a a fully nontrivial solution whose components change sign.
\end{theorem}
 
The proof of the last two theorems and further examples are given in Section \ref{sec:symmetries}. They include, for instance, existence and multiplicity results for \eqref{eq:system} with supercritical nonlinearities ($p>\frac{N}{N-2}$), or in an exterior domain.

Assumption $(A_2)$ may be considerably weakened. As we shall see below, the solution given by Theorem \ref{thm:main_block} minimizes a $\cC^1$-functional $\Psi:\cU\to\r$ defined on an open subset $\cU$ of a smooth Hilbert manifold. So compactness is only needed at the level $c_0:=\inf_\cU\Psi$; see Theorem \ref{thm:block} for the weaker statement.

For $p=2$ our condition $(B_2)$ is basically the same as in \cite[Theorem 1.5]{st} and it is weaker than the one in \cite[Theorem 1.4]{st}. Our approach, however, is different and it has the advantage that it can be used to treat the case $p<2$.

If the system \eqref{eq:system} is purely cooperative (i.e., $q=1$) and $p<2$, assumption $(B_2)$ is satisfied and, so, Theorems \ref{thm:main_block} and \ref{thm:main_fullynontrivial} yield the existence of a fully nontrivial solution. This stands in contrast with the situation for $p=2$ where purely cooperative systems do not always have a positive solution; see, e.g., \cite[Theorem 0.2]{bwa} or \cite[Theorem 1]{s}. For purely cooperative systems with $p=2$ our condition $(B_2)$ is basically that in \cite[Corollary 2.3]{lw}.

Finally, we obtain a new result concerning existence of synchronized solutions when all the $\lambda_i$'s coincide, i.e., $\lambda_i:=\lambda$ for all $i=1,\dots,\ell$. We say that $\bf u=(u_1,\dots,u_\ell)$ is a \emph{fully synchronized} solution if $u_i=c_i u$, where $u$ is a nontrivial solution to  the single equation
$$-\Delta u + \lambda u = |u|^{2p-2}u,\qquad u\in H,$$ 
and  $\bf c=(c_1,\ldots,c_\ell)\in\r^\ell$ solves  the algebraic system
\begin{equation}\label{ss1}
c_i=\sum_{j=1}^\ell \beta_{ij} |c_j|^p|c_i|^{p-2}c_i,\quad c_i>0,\quad\text{for every \ }i=1,\dots,\ell.
\end{equation}
There are some results concerning the solvability of \eqref{ss1}. The easiest case is when $p=2$ and $\ell=2$. Then a solution to \eqref{ss1} exists if and only if
$$\beta_{12}\in (-\sqrt{\beta_{11}\beta_{22}},\min\{\beta_{11},\beta_{22}\} )\cup(\max\{\beta_{11},\beta_{22}\},+\infty).$$
If $p=2$ and $\ell\ge2$,  a solution to \eqref{ss1} exists when  $\beta_{ij}:=\beta$ for all $i\not=j$ and 
 $\beta\in \left(\overline\beta,\min\{\beta_{ii}\}\right)\cup\left(\max\{\beta_{ii}\},+\infty\right)$ for some $\overline\beta<0 $ (see \cite[Proposition 2.1]{b}), while if $p=\frac{N}{N-2}<2$ and $\ell=2$ a solution to \eqref{ss1} always exists provided $\beta_{12}>0$ (see \cite[Theorem 1.1]{cz2}). The following theorem complements these results.

\begin{theorem} \label{thm:synchronized}
Let $p<2$ and assume that the system \eqref{eq:system} is purely cooperative (i.e., $\beta_{ii}>0$ and $\beta_{ij}\geq 0$ for all $i,j=1,\ldots,\ell$, $i\neq j$) and that $\lambda_i=\lambda$ for all $i=1,\ldots,\ell$. Then, for each $i$, there exists $c_i>0$ such that $(c_1 u,\dots,c_\ell u)$ is a solution to \eqref{eq:system} for every solution $u$ to the equation
$$-\Delta u + \lambda u = |u|^{2p-2}u,\qquad u\in H.$$
\end{theorem}

The proof of this result relies on a simple minimization argument and it is given in Section \ref{sec:locked}.

\section{A simple variational approach} \label{sec:variational}

We assume throughout that $(A_1)$ and $(B_1)$ hold true. 

Recall that $H$ is either $H_0^1(\o)$ or $D^{1,2}_0(\o)$. Assumption $(A_1)$ asserts that
$$\|v\|_i:=\Big(\int_\Omega(|\nabla v|^2+\lambda_iv^2)\Big)^{1/2}$$
is a norm in $H$, equivalent to the standard one. Let $H_i$ denote the space $H$ equiped with this norm and, for the partition in assumption $(B_1)$, define
\begin{align*}
\cH_h&:=H_{\ell_{h-1}+1}\times\cdots\times H_{\ell_h},\qquad h=1,\ldots,q,\\
\cH&:=H_1\times\cdots\times H_\ell=\cH_1\times\cdots\times\cH_q.
\end{align*}
A point in $\cH_h$ will be denoted by $\bar u_h$, a point in $\cH$ by
$$\bf u=(\bar u_1,\ldots,\bar u_q)=(u_1,\ldots,u_\ell)\quad\text{with \ }\bar u_h\in\cH_h,\quad u_i\in H_i,$$
and their norms by
$$\|\bar u_h\|:=\Big(\sum_{i\in I_h}\|u_i\|_i^2\Big)^{1/2}\qquad\text{and}\qquad\|\bf u\|:=\Big(\sum_{h=1}^q\|\bar u_h\|^2\Big)^{1/2}.$$
Let $\cJ:\cH\to\r$ be the functional given by 
$$\cJ(u_1,\ldots,u_\ell) := \frac{1}{2}\sum_{i=1}^\ell\|u_i\|_i^2 - \frac{1}{2p}\sum_{i,j=1}^\ell\beta_{ij}\io |u_i|^p|u_j|^p.$$
This functional is of class $\cC^1$ and its critical points are the solutions to the system \eqref{eq:system}. The block-wise nontrivial solutions belong to the set
$$\cN:= \{\bf u\in\cH:\|\bar u_h\|\neq 0\text{ \ and \ }\partial_{\bar u_h}\cJ(\bf u)\bar u_h=0 \text{ \ for all \ } h=1,\ldots,q\}.$$
Note that
\begin{align*}
&\partial_{\bar u_h}\cJ(\bf u)\bar u_h=\|\bar u_h\|^2 - \sum_{(i,j)\in\cI_h}\io\beta_{ij}|u_i|^p|u_j|^p - \sum_{(i,j)\in\cK_h}\io\beta_{ij}|u_i|^p|u_j|^p,
\end{align*}
with $\cI_h$ and $\cK_h$ as defined in assumption $(B_1)$, and that
$$\cJ(\bf u)=(\tfrac{1}{2}-\tfrac{1}{2p})\|\bf u\|^2\qquad\text{if \ }u\in\cN.$$

\begin{definition} \label{def:least energy}
A block-wise nontrivial solution $\bf u$ to the system \eqref{eq:system} such that $\cJ(\bf u)=\inf_\cN \cJ$ will be called a least energy block-wise nontrivial solution.
\end{definition}

In fact, we will show that any minimizer of $\cJ$ on $\cN$ is a critical point of $\cJ$, i.e., a block-wise nontrivial solution to \eqref{eq:system}. We follow the approach introduced in \cite{cs}. 

\begin{lemma} \label{lem:nehari}
\begin{itemize}
\item[$(i)$]There exists $d_0>0$ such that $\min_{h=1,\ldots,q}\|\bar u_h\|^2\geq d_0$ for every $(\bar u_1,\ldots,\bar u_q)\in\cN$. As a consequence, we have that $\cN$ is closed in $\cH$ and \ $\inf_{\bf u\in\cN}\cJ(\bf u)>0$.
\item[$(ii)$]There exists $d_1>0$ independent of $(\beta_{ij})$ such that
$$\inf_{\bf u\in\cN}\cJ(\bf u)\leq d_1\left(\min_{h=1,\ldots,q}\max_{i\in I_h}\beta_{ii}\right)^{-\frac{1}{p-1}}.$$
\end{itemize}
\end{lemma}

\begin{proof}
$(i):$ If $(\bar u_1,\ldots,\bar u_q)\in\cN$, the Hölder and the Sobolev inequalities yield
$$\|\bar u_h\|^2\leq\sum_{(i,j)\in\cI_h}\io\beta_{ij}|u_i|^p|u_j|^p\leq C\|\bar u_h\|^{2p}\quad\text{for every \ }h=1,\ldots,q.$$
Hence, there exists $d_0>0$ such that $\|\bar u_h\|^2\geq d_0$ for every $h=1,\ldots,q$.

$(ii):$ Fix $i_h\in I_h$ such that $\beta_{i_hi_h}=\max\{\beta_{ii}:i\in I_h\}$. Let
$$\cM:=\{(v_1,\ldots,v_q)\in H^q:v_h\neq 0, \ \|v_h\|_{i_h}^2=\io |v_h|^{2p}, \ v_hv_k=0\text{ if }h\neq k\}$$
and define
$$d_1:=(\tfrac{1}{2}-\tfrac{1}{2p})\inf_{(v_1,\ldots,v_q)\in\cM}\sum_{h=1}^q\|v_h\|_{i_h}^2.$$
Given $(v_1,\ldots,v_q)\in\cM$, let $\bar u_h\in\cH_h$ be the function whose $i_h$-th component is $\beta_{i_hi_h}^{-\frac{1}{2p-2}}v_h$ and all other components are $0$. Then, $\bf u=(\bar u_1,\ldots,\bar u_q)\in\cN$ and
\begin{align*}
\inf_{\bf u\in\cN}\cJ(\bf u)&\leq\cJ(\bf u)=(\tfrac{1}{2}-\tfrac{1}{2p})\sum_{h=1}^q\beta_{i_hi_h}^{-\frac{1}{p-1}}\|v_h\|_{i_h}^2\\
&\leq (\tfrac{1}{2}-\tfrac{1}{2p})(\min_h\beta_{i_hi_h})^{-\frac{1}{p-1}}\sum_{h=1}^q\|v_h\|_{i_h}^2,
\end{align*}
and the inequality in $(ii)$ follows.
\end{proof}

Given $\bf u=(\bar u_1,\ldots,\bar u_q)\in\cH$ and $\bf s=(s_1,\ldots,s_q)\in(0,\infty)^q$, we write
$$\bf s\bf u:=(s_1\bar u_1,\ldots,s_q\bar u_q).$$
Let $\cS_h:=\{\bar u\in\cH_h:\|\bar u\|=1\}$ and $\cT:=\cS_1\times\cdots\times\cS_q$. Define 
$$\cU:=\{\bf u\in\cT:\bf s\bf u\in\cN\text{ \ for some \ }\bf s\in(0,\infty)^q\}.$$
Arguing as in \cite[Proposition 3.1 and Theorem 3.3]{cs} one obtains the following two lemmas. We include their proof for the sake of completeness.

\begin{lemma} \label{lem:homeomorphism}
\begin{itemize}
\item[$(i)$] Let $\bf u\in\cT$. If there exists $\bf s_{\bf u}\in(0,\infty)^q$ such that $\bf s_{\bf u}\bf u\in\cN$, then $\bf s_{\bf u}$ is unique and satisfies
$$\cJ(\bf s_{\bf u}\bf u)=\max_{\bf s\in(0,\infty)^q}\cJ(\bf s\bf u).$$
\item[$(ii)$] $\cU$ is an nonempty open subset of $\cT$, and the map \ $\cU\to(0,\infty)^q$ \ given by $\bf u\mapsto\bf s_{\bf u}$ is continuous.
\item[$(iii)$] The map \ $\cU\to \cN$ \ given by \ $\bf u\mapsto\bf s_{\bf u}\bf u$ \ is a homeomorphism.
\item[$(iv)$] If $(\bf u_n)$ is a sequence in $\cU$ and $\bf u_n\to\bf u\in\partial\cU$, then $|\bf s_{\bf u_n}|\to\infty$.
\end{itemize}
\end{lemma}

\begin{proof}
Given $\bf u\in\cH$ we define $J_{\bf u}:(0,\infty)^q \to\r$ by $J_{\bf u}(\bf s):=\cJ(\bf s\bf u)$. Then
$$s_h\,\partial_h J_{\bf u}(\bf s)=\partial_{\bar u_h}\cJ(\bf s\bf u)[s_h \bar u_h],\qquad h=1,\ldots,q.$$
So, if $\|\bar u_h\|\neq 0$ for every $h=1,\ldots,q$, then $\bf s\bf u\in\cN$ iff $\bf s$ is a critical point of $J_{\bf u}$. The function $J_{\bf u}$ can be written as
$$J_{\bf u}(\bf s)= \sum_{h=1}^q a_{\bf u,h}s_h^2 - \sum_{h=1}^q b_{\bf u,h}s_h^{2p} + \sum_{\substack{k=1 \\ k\neq h}}^qd_{\bf u,hk}s_k^ps_h^p,$$
with \ $a_{\bf u,h}:=\frac{1}{2}\|\bar u_h\|^2$, \ $b_{\bf u,h}:=\frac{1}{2p}\sum_{(i,j)\in I_h\times I_h}\beta_{ij}\io |u_i|^p|u_j|^p$ \ and \ $d_{\bf u,hk}:=-\frac{1}{2p}\sum_{(i,j)\in I_h\times I_k}\beta_{ij}\io |u_i|^p|u_j|^p$.
\smallskip

$(i):$ If $\bf u\in\cT$, then $a_{\bf u,h}>0$. Assumption $(B_1)$ implies that $b_{\bf u,h}>0$ and $d_{\bf u,hk}\geq 0$. By \cite[Lemma 2.2]{cs}, if $J_{\bf u}$ has a critical point $\bf s_{\bf u}\in(0,\infty)^\ell$, then it is unique and it is a global maximum of $J_{\bf u}$ in $(0,\infty)^q$.

$(ii):$ Let  $v_1,\ldots,v_q\in H$ be such that $\|v_h\|_{\ell_h}=1$ and $v_h$ and $v_k$ have disjoint supports if $h\neq k$, and let $s_h:=(\beta_{\ell_h,\ell_h}\io|v_h|^{2p})^{-1/(2p-2)}$. Set $\bar u_h:=(0,\ldots,0,v_h)$, $\bf u:=(\bar u_1,\ldots,\bar u_q)$ and $\bf s:=(s_1,\ldots,s_q)$. Then, $\bf u\in\cT$ and $\bf s\bf u\in\cN$. Hence, $\cU\neq\emptyset$.

As $a_{\bf u,h},b_{\bf u,h},d_{\bf u,hk}$ are continuous functions of $\bf u$, \cite[Lemma 2.3]{cs} implies that $\cU$ is open and that the map \ $\cU\to(0,\infty)^q$ \ given by \ $\bf u\mapsto\bf s_{\bf u}$ \ is continuous.

$(iii):$ It follows from $(ii)$ that the map $\cU\to \cN$ given by $\bf u\mapsto\bf s_{\bf u}\bf u$ is  continuous. Its inverse is
$$(\bar u_1,\ldots,\bar u_q) \mapsto\left(\frac {\bar u_1}{\|\bar u_1\|},\ldots,\frac {\bar u_q}{\|\bar u_q\|}\right),$$
which is well defined and continuous. 

$(iv):$ Let $(\bf u_n)$ be a sequence in $\cU$ such that $\bf u_n\to\bf u\in\partial\cU$. If the sequence $(\bf s_{\bf u_n})$ were bounded, after passing to a subsequence we would have $\bf s_{\bf u_n}\to \bf s$. Since $\cN$ is closed, this would imply that $\bf s\bf u\in\cN$ and, hence, that $\bf u\in\cU$. This is impossible because $\cU$ is open in $\cT$. 
\end{proof}

Define $\Psi:\cU\to\r$ by 
\begin{equation} \label{eq:psi}
\Psi(\bf u): = \cJ(\bf s_{\bf u}\bf u)=(\tfrac{1}{2}-\tfrac{1}{2p})|\bf s_{\bf u}|^2.
\end{equation}
As $\cU$ is an open subset of the smooth Hilbert submanifold $\cT$ of $\cH$, we may ask whether $\Psi$ is differentiable. As we shall see below, it is in fact $\cC^1$. We write $\|\Psi'(\bf u)\|_*$ for the the norm of $\Psi'(u)$ in the cotangent space $\mathrm{T}_u^*(\cT)$ to $\cT$ at $u$, i.e.,
$$\|\Psi'(\bf u)\|_*:=\sup\limits_{\substack{\bf v\in\mathrm{T}_{\bf u}(\cU) \\\bf v\neq 0}}\frac{|\Psi'(\bf u)\bf v|}{\|\bf v\|},$$
where $\mathrm{T}_{\bf u}(\cU)$ is the tangent space to $\cU$ at $\bf u$.

Recall that a sequence $(\bf u_n)$ in $\cU$ is called a $(PS)_c$\emph{-sequence for} $\Psi$ if $\Psi(\bf u_n)\to c$ and $\|\Psi'(\bf u_n)\|_*\to 0$, and $\Psi$ is said to satisfy the $(PS)_c$\emph{-condition} if every such sequence has a convergent subsequence. Similarly, a $(PS)_c$\emph{-sequence for} $\cJ$ is a sequence $(\bf u_n)$ in $\cH$ such that $\cJ(\bf u_n)\to 0$ and $\|\cJ'(\bf u_n)\|_{\cH^{-1}}\to 0$, and $\cJ$ satisfies the $(PS)_c$\emph{-condition} if any such sequence has a convergent subsequence.
 
\begin{lemma} \label{lem:psi}
\begin{itemize}
\item[$(i)$] $\Psi\in\cC^1(\cU,\r)$,
\begin{equation*}
\Psi'(\bf u)\bf v = \cJ'(\bf s_{\bf u}\bf u)[\bf s_{\bf u}\bf v] \quad \text{for all } \bf u\in\cU \text{ and }\bf v\in \mathrm{T}_{\bf u}(\cU),
\end{equation*}
and there exists $d_0>0$ such that
$$d_0\,\|\cJ'(\bf s_{\bf u}\bf u)\|_{\cH^{-1}}\leq\|\Psi'(\bf u)\|_*\leq |\bf s_{\bf u}|_\infty\|\cJ'(\bf s_{\bf u}\bf u)\|_{\cH^{-1}}\quad \text{for all } \bf u\in\cU,$$
where $|\bf s|_\infty=\max\{|s_1|,\ldots,|s_q|\}$ if $\bf s=(s_1,\ldots,s_q)$.
\item[$(ii)$] If $(\bf u_n)$ is a $(PS)_c$-sequence for $\Psi$, then $(\bf s_{\bf u_n}\bf u_n)$ is a $(PS)_c$-sequence for $\cJ$. 
\item[$(iii)$] $\bf u$ is a critical point of $\Psi$ if and only if $\bf s_{\bf u}\bf u$ is a critical point of $\cJ$.
\item[$(iv)$] If $(\bf u_n)$ is a sequence in $\cU$ and $\bf u_n\to\bf u\in\partial\cU$, then $\|\Psi(\bf u_n)\|\to\infty$.
\end{itemize}
\end{lemma}

\begin{proof}
$(i):$ Let $\bf u\in\cU$, $\bf v\in\mathrm{T}_{\bf u}(\cU)$, and $\gamma:(-\eps,\eps)\to\cU$ be smooth and such that $\gamma(0)=\bf u$ and $\gamma'(0)=\bf v$. Fix $t\in(-\eps,\eps)$. Recalling that $\cJ(\bf s_{\bf u}\bf u)=\max_{\bf s\in(0,\infty)^q}\cJ(\bf s\bf u)$ and applying the mean value theorem to the function $\tau\mapsto\cJ(\bf s_{\gamma(t)}\gamma'(\tau t))$ we obtain
\begin{align*}
\Psi(\gamma(t))-\Psi(u)& = \cJ(\bf s_{\gamma(t)}\gamma(t)) - \cJ(\bf s_{\bf u}\bf u) \le  \cJ(\bf s_{\gamma(t)}\gamma(t)) - \cJ(\bf s_{\gamma(t)}\bf u)\\
& =t\,\cJ'(\bf s_{\gamma(t)}\gamma(\tau_1t))\,[\bf s_{\gamma(t)}\gamma'(\tau_1t)]
\end{align*}
for some $\tau_1\in(0,1)$. Similarly,
$$\Psi(\gamma(t))-\Psi(u)\geq t\,\cJ'(\bf s_{\bf u}\gamma(\tau_2t))\,[\bf s_{\bf u}\gamma'(\tau_2t)]$$
for some $\tau_2\in(0,1)$. Therefore,
$$\Psi'(\bf u)\bf v=\lim_{t\to 0} \frac{\Psi(\gamma(t))-\Psi(u)}{t} =  \cJ'(\bf s_{\bf u}\bf u)[\bf s_{\bf u}\bf v].$$
It follows that $\Psi$ is of class $\cC^1$ in $\cU$. 

As $\mathrm{T}_{\bf u}(\cU)=\{(\bar v_1,\ldots,\bar v_q)\in\cH:\langle\bar u_h,\bar v_h\rangle=0\text{ for each }h=1,\ldots,q\}$, we have that $\cH = \mathrm{T}_{\bf u}(\cU) \oplus\{(t_1\bar u_1,\ldots,t_q\bar u_q):t_i\in\r\}$ for every $\bf u=(\bar u_1,\ldots,\bar u_q)\in\cU$. Since $\bf s_{\bf u}\bf u\in\mathcal{N}$ and $s_h>0$, we conclude that
$$\sup_{\substack{\bf v\in \mathrm{T}_{\bf u}(\cU) \\ \bf v\neq 0}}\frac{\cJ'(\bf s_{\bf u}\bf u)[\bf s_{\bf u}\bf v]}{\|\bf s_{\bf u}\bf v\|}=\sup_{\substack{\bf w\in\cH \\ \bf w\neq 0}}\frac{\cJ'(\bf s_{\bf u}\bf u)[\bf w]}{\|\bf w\|}.$$
On the other hand, for every $\bf v\in\mathrm{T}_{\bf u}(\cU)$, $\bf v\neq \bf 0$, we have that
$$\min_hs_{\bf u,h}\,\frac{\cJ'(\bf s_{\bf u}\bf u)[\bf s_{\bf u}\bf v]}{\|\bf s_{\bf u}\bf v\|}\leq\frac{|\Psi'(\bf u)\bf v|}{\|\bf v\|}=\frac{\cJ'(\bf s_{\bf u}\bf u)[\bf s_{\bf u}\bf v]}{\|\bf v\|}\leq\max_hs_{\bf u,h}\,\frac{\cJ'(\bf s_{\bf u}\bf u)[\bf s_{\bf u}\bf v]}{\|\bf s_{\bf u}\bf v\|},$$
where $\bf s_{\bf u}=(s_{\bf u,1},\ldots,s_{\bf u,q})$. By Lemma \ref{lem:nehari}, $\min_{h=1,\ldots,q}s_{\bf u,h}\geq d_0$ for every $\bf u\in\cU$.
Taking the supremum over all $\bf v\in\mathrm{T}_{\bf u}(\cU)$, $\bf v\neq \bf 0$, we obtain the inequalities stated in $(i)$.

Statements $(ii)$ and $(iii)$ follow immediately from $(i)$, and statement $(iv)$ follows from Lemma \ref{lem:homeomorphism}$(iv)$ and \eqref{eq:psi}.
\end{proof}

\begin{theorem} \label{thm:block}
Assume $(A_1)$ and $(B_1)$. If $\cJ$ satisfies the $(PS)_{c_0}$-condition at $c_0:=\inf_\cN \cJ$, then the system \eqref{eq:system} has a least energy block-wise nontrivial solution.
\end{theorem}

\begin{proof}
Let $(\bf w_n)$ be a minimizing sequence for $\Psi$ in $\cU$. Lemma \ref{lem:psi}$(iv)$ implies that $\cU$ is positively invariant under the negative pseudogradient flow of $\Psi$, so the deformation lemma \cite[Lemma 5.15]{w} and Ekeland's variational principle \cite[Lemma 8.5]{w} hold true and we may assume that $(\bf w_n)$ is a $(PS)_{c_0}$-sequence for $\Psi$. As $\cJ$ satisfies the $(PS)_{c_0}$-condition, we derive from Lemmas \ref{lem:psi}$(ii)$ and \ref{lem:homeomorphism}$(iii)$ that, passing to a subsequence, $\bf w_n\to\bf w$ and $\bf w$ is a minimum of $\Psi$. Then $\bf u:=\bf s_{\bf w}\bf w$ is a minimum of $\cJ$ on $\cN$ and Lemma \ref{lem:psi}$(iii)$ asserts that it is a critical point of $\cJ$.
\end{proof}
\smallskip

\begin{proof}[Proof of Theorem \ref{thm:main_block}]
A standard argument shows that, if assumption $(A_2)$ holds true, then $\cJ$ satisfies the $(PS)_{c_0}$-condition. So this result follows from Theorem \ref{thm:block}.
\end{proof}

\section{Existence of a fully nontrivial solution} \label{sec:fully nontrivial}

We assume thoughout $(A_1)$ and $(B_1)$. Our aim is to show that, if $(B_2)$ holds true, no semitrivial function in $\cN$ can be a minimizer of $\cJ$ on $\cN$. 

Let $\bf u\in\cN$ be semitrivial. To simplify notation we assume that
\begin{equation} \label{eq:u}
\bf u=(0,u_2,\ldots,u_\ell)=(\bar u_1,\ldots,\bar u_q)\qquad\text{with \ }u_i\in H_i,\quad\bar u_h\in\cH_q.
\end{equation}
Given $\vp\in H$ and $\eps>0$ define
$$\bf u_\eps:=(\bar u_{\eps,1},\bar u_2,\ldots,\bar u_q)\quad\text{with \ }\bar u_{\eps,1}:=(\eps\vp,u_2,\ldots,u_{\ell_1})\in\cH_1.$$
Set \ $\hat I_1:=I_1\smallsetminus\{1\}$ \ and \ $\hat I_h:=I_h$ if $h=2,\ldots,q$.

\begin{lemma} \label{lem:fully_nontrivial1}
There exist $\eps_0>0$ and a $\cC^1$-map $\bf t:(-\eps_0,\eps_0)\to(0,\infty)^q$ satisfying \ $\bf t(\eps)\bf u_\eps\in\cN$, $\bf t(0)=(1,\dots,1)$ and $\bf t'(0)= (0,\dots,0)$.
\end{lemma}

\begin{proof}
For $\eps\in\r$ and $\bf t=(t_1,\ldots,t_q)\in (0,\infty)^q$ define
\begin{align*}
&F_1(\eps,\bf t):=\partial_{\bar u_1}\cJ(\bf t\bf u_\eps)[t_1\bar u_{\eps,1}]\\
&\qquad=t_1^2\eps^2\|\vp\|_1^2 + t_1^2\|\bar u_1\|^2 -  t_1^{2p}\eps^{2p} \beta_{11}\io|\varphi|^{2p}\\
&\qquad\quad -2t_1^{2p}|\eps|^p\sum_{j\in \hat I_1} \beta_{1j}\io|\varphi|^p|u_j|^p -t_1^{2p}\sum_{(i,j)\in \hat I_1\times \hat I_1} \beta_{ij}\io |u_i|^p|u_j|^p \\
&\qquad\quad - t_1^p|\eps|^p\sum_{k=2}^q \sum_{j\in I_k }t_k^p\beta_{1j}\io|\vp|^p|u_j|^p -t_1^p\sum_{k=2}^q t_k^p\sum_{(i,j)\in\hat I_1\times I_k} \beta_{ij}\io |u_i|^p|u_j|^p,
\end{align*}
and for $h=2,\ldots,q$,
\begin{align*}
&F_h(\eps,\bf t):=\partial_{\bar u_h}\cJ(\bf t\bf u_\eps)[t_h\bar u_h]\\
&\qquad=t_h^2\|\bar u_h\|^2 -t_h^{2p}\sum_{(i,j)\in I_h\times I_h}\beta_{ij}\io|u_i|^p|u_j|^p
-  t_h^p t_1^p|\eps|^p\sum_{i\in I_h} \beta_{i1}\io|u_i|^p|\vp|^p \\ 
&\qquad\quad - t_h^p t_1^p\sum_{(i,j)\in I_h\times \hat I_1} \beta_{ij}\io|u_i|^p|u_j|^p
-t_h^p\sum_{\substack{k=2\\ k\neq h}}^q t_k^p\sum_{(i,j)\in I_h\times I_k}\beta_{ij}\io|u_i|^p|u_j|^p.
\end{align*}
We shall apply the implicit function theorem to the $\cC^1$-map $\bf F=(F_1,\ldots,F_q):\r\times (0,\infty)^q\to\r^q.$ Set $\bf 1=(1,\ldots,1)$. Note that, as $\bf u\in\cN$,
\begin{align} \label{eq:F(0,1)}
F_h(0,\bf 1)&=\|\bar u_h\|^2-\sum_{(i,j)\in \hat I_h\times \hat I_h} \beta_{ij}\io |u_i|^p|u_j|^p-\sum_{\substack{k=1\\ k\neq h}}^q \sum_{(i,j)\in \hat I_h\times\hat I_k} \beta_{ij}\io|u_i|^p|u_j|^p \\
&=\partial_{\bar u_h}\cJ(\bf u_\eps)[\bar u_h]=0\qquad\text{for every \ }h=1,\ldots,q. \nonumber
\end{align}
Using this identity we obtain
\begin{align} \label{eq:a_hh}
&a_{hh}:=\partial_{t_h} F_h(0,\bf 1)\\
&=2\|\bar u_h\|^2-2p\sum_{(i,j)\in \hat I_h\times\hat I_h} \beta_{ij}\io|u_i|^p|u_j|^p -p\sum_{\substack{k=1\\ k\neq h}}^q \sum_{(i,j)\in \hat I_h\times\hat I_k }\beta_{ij}\io|u_i|^p|u_j|^p \nonumber \\
&=(2-2p)\sum_{(i,j)\in \hat I_h\times\hat I_h} \beta_{ij}\io|u_i|^p|u_j|^p + (2-p)\sum_{\substack{k=1\\ k\neq h}}^q \sum_{(i,j)\in \hat I_h\times\hat I_k }\beta_{ij}\io|u_i|^p|u_j|^p \nonumber \\
&<p \ \sum_{\substack{k=1\\ k\neq h}}^q \sum_{(i,j)\in \hat I_h\times\hat I_k }\beta_{ij}\io|u_i|^p|u_j|^p\leq 0. \nonumber
\end{align}
Furthermore, we have
\begin{equation} \label{eq:a_hk}
a_{hk}:=\frac{\partial F_h}{\partial t_k}(0,\bf 1)= -p\sum\limits_{(i,j)\in \hat I_h\times\hat I_k } \beta_{ij}\io|u_i|^p|u_j|^p>0\qquad\text{if \ }h\neq k.
\end{equation}
The Jacobian matrix of $\bf F$ with respect to $\bf t$ at the point $(0,\bf 1)$ is strictly diagonally dominant, i.e. 
$$|a_{hh}|>\sum_{h\neq k}|a_{hk}|\qquad \text{for every}\ h=1,\ldots,q.$$
Indeed, \eqref{eq:F(0,1)} yields
\begin{align*}
&|a_{hh}|-\sum_{h\neq k}|a_{hk}|=-a_{hh}-\sum_{h\neq k}a_{hk}\\
&=(2p-2)\sum_{(i,j)\in \hat I_h\times\hat I_h} \beta_{ij}\io|u_i|^p|u_j|^p + (p-2)\sum_{\substack{k=1\\ k\neq h}}^q \sum_{(i,j)\in \hat I_h\times\hat I_k }\beta_{ij}\io|u_i|^p|u_j|^p\\
&\quad + p\sum_{\substack{k=1\\ k\neq h}}^q\sum\limits_{(i,j)\in \hat I_h\times\hat I_k } \beta_{ij}\io|u_i|^p|u_j|^p \\
&=(2p-2)\,\|\bar u_h\|^2>0.
\end{align*}
The Levy-Desplanques theorem asserts that a strictly diagonally dominant matrix is non-singular. So, by the implicit function theorem, there exist $\eps_0>0$ and a $\cC^1$-map $\bf t:(-\eps_0,\eps_0)\to(0,\infty)^q$ satisfying $\bf t(0)=\bf 1$,
\begin{equation*}
F_h(\eps,\bf t(\eps))=0\qquad\text{for every \ }\eps\in (-\eps_0,\eps_0), \ h=1,\ldots,q,
\end{equation*}
i.e., $\bf t(\eps)\bf u_\eps\in\cN$, and $\bf t'(\eps)=-\partial_{\bf t}\bf F(\eps,\bf t(\eps))^{-1}\circ\partial_\eps\bf F(\eps,\bf t(\eps))$. As
\begin{align} \label{eq:partial_eps}
 &\partial_\eps F_1(\eps,\bf t)=2\eps t_1^2 \|\vp\|_1^2 - 2p\eps^{2p-1} t_1^{2p}\beta_{11}\io|\vp|^{2p}\\
 &\quad-(2p)\,|\eps|^{p-2}\eps\, t_1^{2p}\sum_{j\in\hat I_1}\beta_{1j}\io|\vp|^p|u_j|^p  - p\,|\eps|^{p-2}\eps\, t_1^{p}\sum_{k=2}^q\sum_{j\in I_k }t_k^p\beta_{1j}\io|\vp|^p|u_j|^p \nonumber \\
 & \partial_\eps F_h(\eps,\bf t) = -p\,|\eps|^{p-2}\eps\, t_h^pt_1^p \sum_{i\in I_h}\beta_{i1}\io|u_i|^p|\vp|^p\qquad\text{if \ }h=2,\ldots,q, \nonumber
\end{align}
we conclude that $\bf t'(0)=\bf 0$, as claimed.
\end{proof}

\begin{remark} \label{rem:p=2}
\emph{
For $p=2$ the map $\bf t$ is of class $\cC^2$ and its second derivative at $0$ solves the system $\partial_{\bf t}\bf F(0,\bf 1)\bf t''(0)+\partial_\eps\bf F(0,\bf 1)=0$, i.e.,
\begin{equation*}
\begin{cases}
\sum\limits_{k=1}^q a_{1k}t''_k(0)=-2 \|\vp\|_1^{2}+ 4\sum\limits_{j\in \hat I_1}\beta_{1j}\io|\vp|^{2}|u_j|^{ 2} +2\sum\limits_{k=2}^q\sum\limits_{j\in \hat I_k}\ \beta_{1j} \io|\vp|^{2}|u_j|^{ 2}\\
\sum\limits_{k=1}^q a_{hk}t''_k(0)=   2 \sum\limits_{j\in \hat I_h}\ \beta_{1j} \io|\varphi|^{2}|u_j|^{ 2}\qquad \text{if } \ h=2,\dots,q,
\end{cases}
\end{equation*}
with $a_{hk}$ as in the previous lemma, i.e.,
$$a_{hk}=-2\sum\limits_{(i,j)\in \hat I_h\times\hat I_k}\beta_{ij}\io|u_i|^2|u_j|^2.$$
As a consequence,
$$\sum_{h,k=1}^qa_{hk}t''_k(0)=-2 \|\vp\|_1^{2}+4\sum_{h=1}^q\sum_{j\in\hat I_h}\beta_{1j}\io|\vp|^2|u_j|^2.$$
}
\end{remark}
\medskip

\begin{lemma} \label{lem:fully_nontrivial2}
If $p<2$ then, for small enough $\eps>0$,
\begin{equation*}
\cJ(\bf t(\eps)\bf u_\eps)-\cJ(\bf u)=\eps^p \Big(-\frac{1}{p}\sum_{k=1}^q \sum_{j\in\hat I_k }\beta_{1j}\io|\vp|^p|u_j|^p +o(1)\Big).
\end{equation*}
\end{lemma}

\begin{proof}
Since $\bf u\in\cN$ and $\bf t(\eps)\bf u_\eps\in\cN$ we have that
\begin{equation*}
\cJ(\bf u)=\Big(\frac{1}{2}-\frac{1}{2p}\Big)\Big(\sum_{h,k=1}^q \sum_{(i,j)\in \hat I_h\times\hat I_k} \beta_{ij}\io|u_i|^p|u_j|^p\Big)
\end{equation*}
and, writing $\bf t(\eps)=(t_1,\ldots,t_q)$ where $t_h=t_h(\eps)$, 
\begin{align*}
\cJ(\bf t(\eps)\bf u_\eps)&=\Big(\frac{1}{2}-\frac{1}{2p}\Big)\Big(t_1^{2p}\beta_{11}\io|\eps\vp|^{2p} +2  \sum_{k=1}^q t_1^pt_k^p\sum_{j\in\hat I_k}\beta_{1j}\io|\eps\vp|^p|u_j|^p \\ 
&\quad + \sum_{h,k=1}^qt_h^pt_k^p\sum_{(i,j)\in\hat I_h\times\hat I_k}\beta_{ij}\io|u_i|^p|u_j|^p\Big).
\end{align*}
Therefore,
\begin{align} \label{eq:difference}
\cJ(\bf t(\eps)\bf u_\eps)-\cJ(\bf u)&=\frac{p-1}{2p}\Big(t_1^{2p}\beta_{11}\io|\eps\vp|^{2p} +2  \sum_{k=1}^q t_1^pt_k^p\sum_{j\in\hat I_k}\beta_{1j}\io|\eps\vp|^p|u_j|^p \\ 
&\qquad\qquad \ + \sum_{h,k=1}^q(t_h^pt_k^p-1)\sum_{(i,j)\in\hat I_h\times\hat I_k}\beta_{ij}\io|u_i|^p|u_j|^p\Big) \nonumber \\
&=\frac{p-1}{2p}\Big(2\eps^p\sum_{k=1}^q \sum_{j\in\hat I_k}\beta_{1j}\io|\vp|^p|u_j|^p + o(\eps^p) \nonumber \\
&\qquad\qquad \ + \sum_{h,k=1}^q(t_h^pt_k^p-1)\sum_{(i,j)\in\hat I_h\times\hat I_k}\beta_{ij}\io|u_i|^p|u_j|^p\Big). \nonumber
\end{align}
Set
\begin{equation} \label{eq:b_hk}
b_{hk}:=\sum_{(i,j)\in\hat I_h\times\hat I_k}\beta_{ij}\io|u_i|^p|u_j|^p\qquad h,k=1,\ldots,q.
\end{equation}
To estimate \eqref{eq:difference} we need to expand the function
\begin{equation*}
\sum_{h,k=1}^q(t_h^p(\eps)t_k^p(\eps)-1)\,b_{hk}.
\end{equation*}
As $p<2$ we derive from \eqref{eq:partial_eps} that
\begin{align} \label{eq:T1}
&T_1:=\lim_{\eps\to 0^+}\frac{\partial_\eps F_1(\eps,\bf t(\eps))}{\eps^{p-1}} = -2p\sum_{j\in\hat I_1}\beta_{1j}\io|\vp|^p|u_j|^p  - p \sum_{k=2}^q\sum_{j\in I_k }\beta_{1j}\io|\vp|^p|u_j|^p
\\ 
&T_h:=\lim_{\eps\to 0^+}\frac{\partial_\eps F_h(\eps,\bf t(\eps))}{\eps^{p-1}} = -p\sum_{i\in I_h}\beta_{i1}\io|u_i|^p|\vp|^p\qquad\text{if \ }h=2,\ldots,q.  \nonumber
\end{align}
So, since $\partial_{\bf t}\bf F(\eps,\bf t(\eps))\bf t'(\eps)+\partial_\eps\bf F(\eps,\bf t(\eps))=\bf 0$, we conclude that
\begin{equation} \label{eq:tau1}
\lim_{\eps\to 0^+}\frac{t'_h(\eps)}{\eps^{p-1}}=\tau_h\in\r\qquad\text{for every \ }h=1,\ldots,q,
\end{equation}
and
\begin{equation} \label{eq:tau2}
\partial_{\bf t}\bf F(0,\bf 1)\boldsymbol\tau + \bf T =\bf 0,
\end{equation}
where $\boldsymbol\tau=(\tau_1,\ldots,\tau_q)$ and $\bf T=(T_1,\ldots,T_q)$.	
From \eqref{eq:tau1} and L'Hôpital's rule we get
\begin{equation} \label{eq:t^p}
t_h^p(\eps)\,t_k^p(\eps)-1=(\tau_h+\tau_k)\,\eps^p + o(\eps^p),
\end{equation}
and from \eqref{eq:tau2}, \eqref{eq:a_hh}, \eqref{eq:a_hk} and \eqref{eq:b_hk} we derive 
\begin{align} \label{eq:T2}
-\sum_{h=1}^qT_h &=\sum_{h,k=1}^qa_{hk}\tau_k \\
&=(2-2p)\sum_{h=1}^qb_{hh} + (2-p)\sum_{\substack{h,k=1 \\ k\neq h}}^q b_{hk}\tau_h-p\sum_{\substack{h,k=1 \\ k\neq h}}^q b_{hk}\tau_h \nonumber \\
&=2(1-p)\sum_{h,k=1}^q b_{hk}\tau_h. \nonumber
\end{align}
Now, \eqref{eq:t^p}, \eqref{eq:T2} and \eqref{eq:T1} yield
\begin{align*}
\sum_{h,k=1}^q(t_h^p(\eps)\,t_k^p(\eps)-1)\,b_{hk}&=\eps^p\Big(\sum_{h,k=1 }^q(\tau_h+\tau_k)b_{hk}+o(1)\Big)\\
&=\eps^p\Big(\frac{1}{p-1}\sum_{h=1}^qT_h+o(1)\Big)\\
&=\eps^p\Big(-\frac{2p}{p-1}\sum_{h=1}^q\sum_{j\in\hat I_h}\beta_{1j}\io|\vp|^p|u_j|^p + o(1)\Big).
\end{align*}
Going back to \eqref{eq:difference} we conclude that
\begin{align*}
\cJ(\bf t(\eps)\bf u_\eps)-\cJ(\bf u)=\eps^p\Big(-\frac{1}{p}\sum_{h=1}^q\sum_{j\in\hat I_h}\beta_{1j}\io|\vp|^p|u_j|^p + o(1)\Big),
\end{align*}
as claimed.
\end{proof}

\begin{lemma} \label{lem:fully_nontrivial3}
If $p=2$ then, for small enough $\eps>0$,
\begin{equation*}
\cJ(\bf t(\eps)\bf u_\eps)-\cJ(\bf u)=\eps^2\Big(\|\vp\|_1^2-\sum_{k=1}^q \sum_{j\in\hat I_k }\beta_{1j}\io|\vp|^2|u_j|^2 +o(1)\Big).
\end{equation*}
\end{lemma}

\begin{proof}
As $\bf t$ is of class $\cC^2$ when $p=2$, from  and Taylor's expansion for $t_h^2t_k^2$ we derive
\begin{align} \label{eq:difference2}
\cJ(\bf t(\eps)\bf u_\eps)-\cJ(\bf u)&=\eps^2\Big(\sum_{k=1}^q \sum_{j\in\hat I_k}\beta_{1j}\io|\vp|^2|u_j|^2 + o(1) \nonumber \\
&\qquad\quad + \frac{1}{2}\sum_{h,k=1}^q(t''_h(0)+t''_k(0))\sum_{(i,j)\in\hat I_h\times\hat I_k}\beta_{ij}\io|u_i|^2|u_j|^2\Big). \nonumber
\end{align}
Now we use Remark \ref{rem:p=2}  to compute
\begin{align*}
&\sum_{k=1}^q \sum_{j\in\hat I_k}\beta_{1j}\io|\vp|^2|u_j|^2+ \frac{1}{2}\sum_{h,k=1}^q(t''_h(0)+t''_k(0))\sum_{(i,j)\in\hat I_h\times\hat I_k}\beta_{ij}\io|u_i|^2|u_j|^2 \\
&=\sum_{k=1}^q \sum_{j\in\hat I_k}\beta_{1j}\io|\vp|_1^2|u_j|^2 - \frac{1}{2}\sum_{h,k=1}^qa_{hk}t''_k(0) \\
&=\|\vp\|_1^{2}-\sum_{h=1}^q\sum_{j\in\hat I_h}\beta_{1j}\io|\vp|^2|u_j|^2.
\end{align*}
This completes the proof.
\end{proof}

\medskip

\begin{proof}[Proof of Theorem \ref{thm:main_fullynontrivial}]
Let $\bf u\in\cN$ be such that $\cJ(\bf u)=\inf_\cN\cJ$. Set
$$S:=\min_{i=1,\ldots,\ell}\,\inf_{\substack{v\in H \\ v\neq 0}}\frac{\|v\|_i^2}{ \ |u|_{2p}^2},\qquad\text{where \ }|v|_{2p}:=\big(\io |v|^{2p}\Big)^\frac{1}{2p},$$
and define $C_*:=\Big(\frac{pd_1}{(p-1)S^\frac{p}{p-1}}\Big)^p$, with $d_1$ is as in Lemma \ref{lem:nehari}. We distinguish two cases.

1) Let $p<2$. Arguing by contradiction, assume that that $(B_2)$ holds true and that some component of $\bf u$ is trivial. Without loss of generality, we assume it is the first one, i.e., $\bf u$ is as in \eqref{eq:u}. Fix $i^*_1\in \hat I_1$ such that $\max_{i\in\hat I_1}|u_i|_{2p}=|u_{i_1^*}|_{2p}$. As $\bf u\in\cN$ we have that $u_{i_1^*}\neq 0$ and using Hölder's inequality we obtain
\begin{align*}
S|u_{i_1^*}|_{2p}^2\leq \|u_{i_1^*}\|_{i_1^*}^2\leq \sum_{i,j\in\hat I_1}\beta_{ij}\io|u_i|^p|u_j|^p\leq\hat\ell_1^2\max_{i,j\in\hat I_1}\beta_{ij}\,|u_{i_1^*}|_{2p}^{2p},
\end{align*}
with $\hat\ell_1:=\ell_1-1$. Therefore,
\begin{align*}
\sum_{j\in\hat I_1 }\beta_{1j}\io|u_{i_1^*}|^p|u_j|^p\geq \beta_{1i_1^*}\io|u_{i_1^*}|^{2p}\geq \min_{\substack{i,j\in I_1 \\ i\neq j}}\beta_{ij}\Big(\frac{S}{\hat\ell_1^2\,\max_{i,j\in I_1}\beta_{ij}}\Big)^\frac{p}{p-1}.
\end{align*}
On the other hand,
\begin{align*}
S^2\Big(\io|u_{i_1^*}|^p|u_j|^p\Big)^\frac{2}{p}\leq S^2|u_{i_1^*}|^2_{2p}|u_j|^2_{2p}\leq\|u_{i_1^*}\|^2_{_{i_1^*}}\|u_j\|^2_{j}.
\end{align*}
As $\cJ(\bf u)=\inf_\cN\cJ$ Lemma \ref{lem:nehari} yields
$$\tfrac{p-1}{p}\|u_j\|^2_{j}\leq \tfrac{p-1}{p}\|\bf u\|^2=\cJ(\bf u)=\inf_\cN\cJ\leq d_1\Big(\min_{h=1,\ldots,q}\max_{i\in I_h}\beta_{ii}\Big)^{-\frac{1}{p-1}}.$$
Therefore, 
\begin{equation} \label{eq:ub2}
\io|u_{i_1^*}|^p|u_j|^p\leq\Big(\frac{pd_1}{(p-1)S}\Big)^p\Big(\min_{h}\max_{i\in I_h}\beta_{ii}\Big)^{-\frac{p}{p-1}}.
\end{equation}
Set $\vp:=u_{i_1^*}$ and let $\cK_1$ be as in $(B_1)$. From $(B_2)$ with $h=1$ we get
\begin{align*}
&\sum_{k=1}^q \sum_{j\in\hat I_k }\beta_{1j}\io|\vp|^p|u_j|^p=\sum_{j\in\hat I_1 }\beta_{1j}\io|u_{i_1^*}|^p|u_j|^p + \sum_{k=2}^q \sum_{j\in\hat I_k }\beta_{1j}\io|u_{i_1^*}|^p|u_j|^p \\
&\geq \min_{\substack{i,j\in I_1 \\ i\neq j}}\beta_{ij}\Big(\frac{S}{\hat\ell_1^2\,\max_{i,j\in I_1}\beta_{ij}}\Big)^\frac{p}{p-1} - \sum_{(i,j)\in\cK_1}|\beta_{ij}|\Big[\frac{pd_1}{(p-1)S}\Big]^p\Big(\min_{h}\max_{i\in I_h}\beta_{ii}\Big)^{\frac{-p}{p-1}} \\
&=\Big(\frac{\hat\ell_1^2}{S}\min_{h}\max_{i\in I_h}\beta_{ii}\Big)^{\frac{-p}{p-1}}\left(\min_{\substack{i,j\in I_1 \\ i\neq j}}\beta_{ij}\left[\frac{\min\limits_{h}\max\limits_{i\in I_h}\beta_{ii}}{\max\limits_{i,j\in I_1}\beta_{ij}}\right]^\frac{p}{p-1}-C_*\hat\ell_1^{\,\frac{2p}{p-1}}\sum_{(i,j)\in\cK_1}|\beta_{ij}|\right) \\
&>0.
\end{align*}
Lemma \ref{lem:fully_nontrivial2} asserts that $\bf t(\eps)\bf u_\eps\in\cN$ for small enough $\eps>0$, and from Lemma \ref{lem:fully_nontrivial1} and the previous inequality we derive that
\begin{equation*}
\cJ(\bf t(\eps)\bf u_\eps)<\cJ(\bf u)=\inf_\cN\cJ.
\end{equation*}
This is a contradiction.
\smallskip

2) Let $p=2$. Arguing again by contradiction, assume that that $(B_2)$ holds true and that $\bf u$ is as in \eqref{eq:u}. Fix $i^*_1\in \hat I_1$ such that $\max_{i\in I_1}|u_i|_4=|u_{i_1^*}|_4$. Let us assume, for simplicity, that $i^*_1=2$. Then $u_2\neq 0$. By $(B_2)$ we have that $\beta_{ij}=b_1$ for $i,j\in I_1$ with $i\neq j$ and that $\beta_{22}\leq b_1$. Therefore,
$$S|u_2|_4^2\leq\|\bf u\|^2\leq\sum_{i\in\hat I_1}\beta_{2i}\io|u_2|^2|u_i|^2\leq(\ell_1-1)b_1|u_2|_4^4,$$
and, so, 
$$\frac{S}{(\ell_1-1)b_1}\leq |u_2|_4^2.$$
On the other hand, \eqref{eq:ub2} reads
$$\io|u_2|^2|u_j|^2\leq\Big(\frac{2d_1}{S}\Big)^2\Big(\min_{h}\max_{i\in I_h}\beta_{ii}\Big)^{-2}.$$
Now, since $\bf u$ solves the system \eqref{eq:system}, we know that 
$$\|u_{2}\|_2^2=\sum_{h=1}^q\sum_{j\in \hat I_h}\ \beta_{2j} \io|u_{2}|^{2}|u_j|^{ 2}.$$
Set $\vp:=u_2$. As $\lambda_1=\lambda_2$ by $(B_2)$, we have that $\|\vp\|_1=\|u_2\|_2$. Therefore,
\begin{align*}
&\|\vp\|_1^2- \sum_{h=1}^q \sum_{j\in \hat I_h}\ \beta_{1j} \io|u_{2}|^{2}|u_j|^{ 2}\\ 
&=\sum_{h=1}^q\sum_{j\in \hat I_h}\ \beta_{2j} \io|u_{2}|^{2}|u_j|^{ 2}-\sum_{h=1}^q\sum_{j\in \hat I_h}\ \beta_{1j} \io|u_{2}|^{2}|u_j|^{ 2}\\
& =(\beta_{22} - \beta_{12}) \io|u_{2}|^{4}+\sum_{j\in \hat I_1\smallsetminus\{2\}}(\beta_{2j} - \beta_{1j}) \io|u_{2}|^{2}|u_j|^{ 2}\\
&\qquad +\sum_{h=2}^q\sum_{j\in I_h}(\beta_{2j}- \beta_{1j})\io|u_{2}|^{2}|u_j|^{ 2}\\ 
&
\le \Big(\max\limits_{i\in I_1}\beta_{ii}-b_1\Big)\left[\frac{S}{(\ell_1-1)b_1}\right]^2+\Big[\frac{2d_1}{S}\Big]^2\Big(\min_{h}\max_{i\in I_h}\beta_{ii}\Big)^{-2}\sum_{\substack{j\in I_h\\h\geq 2}}|\beta_{2j}- \beta_{1j}| \\
&=\left[\frac{S}{(\ell_1-1)b_1}\right]^2\left(\max\limits_{i\in I_1}\beta_{ii}-b_1 + C_*\Big(\frac{b_1(\ell_1-1)}{\min_{h}\max_{i\in I_h}\beta_{ii}}\Big)^2\sum_{\substack{j\in I_h\\h\geq 2}}|\beta_{2j}- \beta_{1j}|\right)\\
&<0,
\end{align*}
by $(B_2)$. Now Lemma \ref{lem:fully_nontrivial3} yields a contradiction.
\end{proof}

\section{Systems with symmetries}
\label{sec:symmetries}

Let $G$ be a closed subgroup of the group $O(N)$ of linear isometries of $\rn$ and denote by $Gx:=\{gx:g\in G\}$ the $G$-orbit of $x\in\rn$. Assume that $\o$ is $G$-invariant, i.e., $Gx\subset\o$ for every $x\in\o$. Then, a function $u:\o\to\r$ is called \emph{$G$-invariant} if it is constant on $Gx$ for every $x\in\o$. Define 
$$H^G:=\{u\in H: u\text{ is }G\text{-invariant}\},$$
where, as before, $H$ is either $H_0^1(\o)$ or $D_0^{1,2}(\o)$.

Further, let $\phi:G\to\z_2:=\{-1,1\}$ be a continuous homomorphism of groups and let $K:=\ker\phi$ be its kernel. Assume
\begin{itemize}
\item[$(\phi)$] there exists $x_0\in\o$ such that $Kx_0\neq Gx_0$.
\end{itemize}
A function $u:\o\to\r$ is called \emph{$\phi$-equivariant} if $u(gx)=\phi(g)u(x)$ for every $g\in G$ and $x\in\o$. Define 
$$H^\phi:=\{u\in H: u\text{ is }\phi\text{-equivariant}\}.$$
Assumption $(\phi)$ guarantees that this space has infinite dimension. Moreover, it implies that $K\neq G$, i.e., that $\phi$ is surjective. Therefore, every nontrivial $\phi$-equivariant function is nonradial and changes sign. 

By the principle of symmetric criticality \cite[Theorem 1.28]{w} the critical points of the restriction of $\cJ$ to the either $(H^G)^\ell$ or to $(H^\phi)^\ell$ are critical poins of $\cJ$, i.e., they solve system \eqref{eq:system}. Clearly, all results in the previous sections go through if we take $\cH$ to be one of these spaces. So Theorems \ref{thm:block}, \ref{thm:main_block} and \ref{thm:main_fullynontrivial} hold true for $H^G$ and $H^\phi$ as well.

Next, we give some applications.

\subsection{Systems in bounded domains}

Let $G$ be a closed subgroup of $O(N)$ and $\o$ a $G$-invariant domain. Let $\lambda_1^G(\o)$ be the first eigenvalue of $-\Delta$ in $H_0^1(\o)^G$ and $d:=\min\{\dim (Gx):x\in\o\}$. Then we have the following result. 

\begin{theorem} \label{thm:general bounded}
If $\o$ is bounded, $\lambda_i>-\lambda_1^G(\o)$, $1<p<\frac{N-d}{N-d-2}$ when $N>d+2$, and $(\beta_{ij})$ satisfies $(B_1)$ and $(B_2)$, the system
\begin{equation} \label{eq:system1}
\begin{cases}
-\Delta u_i+ \lambda_iu_i = \sum\limits_{j=1}^\ell \beta_{ij}|u_j|^p|u_i|^{p-2}u_i, \\
u_i\in H_0^1(\o),\qquad i=1,\ldots,\ell,
\end{cases}
\end{equation}
has a fully nontrivial solution whose components are positive and $G$-invariant. 

Moreover, if there exists a continuous homomorphism of groups $\phi:G\to\z_2$ satisfying $(\phi)$, then \eqref{eq:system1} has has a fully nontrivial solution whose components are $\phi$-equivariant and, thus, change sign.
\end{theorem}

\begin{proof}
It is shown in \cite[Corollary 2]{hv} that, if $\o$ is bounded and $1<p<\frac{N-d}{N-d-2}$, the embedding $H_0^1(\o)^G\hookrightarrow L^p(\o)$ is compact. So Theorems \ref{thm:main_block} and \ref{thm:main_fullynontrivial} yield a least energy fully nontrivial solution $(u_1,\ldots,u_\ell)$ with $u_i \in H_0^1(\o)^G$. As $|u|\in H_0^1(\o)^G$ for every $u\in H_0^1(\o)^G$, $(|u_1|,\ldots,|u_\ell|)$ is also a least energy fully nontrivial solution.

On the other hand, since $\dim (Kx)=\dim (Gx)$ for all $x\in\o$, $H_0^1(\o)^K\hookrightarrow L^p(\o)$ is also compact and, as $H_0^1(\o)^\phi\subset H_0^1(\o)^K$, the embedding $H_0^1(\o)^\phi\hookrightarrow L^p(\o)$ is compact. So Theorems \ref{thm:main_block} and \ref{thm:main_fullynontrivial} yield a least energy fully nontrivial solution $(u_1,\ldots,u_\ell)$ with $u_i \in H_0^1(\o)^\phi$. Since $\phi$ is surjective, $u_i$ changes sign.
\end{proof}

Note that this result includes the case when $\o$ has no symmetries. Then $d=0$ and the system is subcritical. Observe also that the system is supercritical if $d\geq 1$.

For highly symmetric domains Theorem \ref{thm:general bounded} yields infinitely many solutions.

\begin{theorem} \label{thm:ball}
If $\o$ is a ball or an annulus, $\lambda_i>-\lambda_1(\o)$, $1<p<\frac{N}{N-2}$ and $(\beta_{ij})$ satisfies $(B_1)$ and $(B_2)$, the system \eqref{eq:system1} has infinitely many fully nontrivial solutions. All components of one of them are radial and positive, and all components of the rest are nonradial and change sign.
\end{theorem}

\begin{proof}
For each $m\in\n$ let $G_m$ be the group generated by $\vartheta_m:=$ the rotation by $\frac{\pi}{2^m}$ on $\r^2$ acting on the first factor of $\r^2\times\r^{N-2}\equiv\rn$, and let $\phi_m:G_m\to\z_2$ be the homomorphism given by $\phi_m(\vartheta_m):=-1$. Theorem \ref{thm:general bounded} applied to $O(N)$ yields a fully nontrivial solution whose components are positive and radial, and applied to $\phi_m$ yields a fully nontrivial solution whose components $u_{m,i}$ satisfy
$$u_{m,i}(\vartheta_m x,y)=-u_{m,i}(x,y)\qquad\text{for every \ }(x,y)\in\r^2\times\r^{N-2}.$$
Hence, $u_{m,i}$ is nonradial and it changes sign. It is easy to see that $u_{m,i}\neq u_{n,i}$ if $m\neq n$.
\end{proof}

\subsection{Subcritical systems in exterior domains}

Let $G$ be a closed subgroup of $O(N)$ such that the $G$-orbit $Gx$ of every point $x\in\rn\smallsetminus\{0\}$ is an infinite set, and let $\o$ be a $G$-invariant exterior domain (i.e., $\rn\smallsetminus\o$ is bounded, possibly empty). Under these assumptions we have the following result.

\begin{theorem} \label{thm:exterior}
If $\lambda_i>0$, $1<p<\frac{N}{N-2}$ and $(\beta_{ij})$ satisfies $(B_1)$ and $(B_2)$, the system \eqref{eq:system1} has a fully nontrivial solution whose components are positive and $G$-invariant. 

If, in addition, there exists a continuous homomorphism of groups $\phi:G\to\z_2$ satisfying $(\phi)$, then \eqref{eq:system1} has has a fully nontrivial solution whose components are $\phi$-equivariant and, thus, change sign.
\end{theorem}

\begin{proof}
Since the $G$-orbit of every point $x\in\rn\smallsetminus\{0\}$ is an infinite set, the embedding $H_0^1(\o)^G\hookrightarrow L^p(\o)$ is compact \cite[Lemma 4.3]{cs}. If $K=\ker \phi$, the $K$-orbit of every $x\in\rn\smallsetminus\{0\}$ is also infinite. Hence, $H_0^1(\o)^K\hookrightarrow L^p(\o)$ is also compact and so is $H_0^1(\o)^\phi\hookrightarrow H_0^1(\o)^K\hookrightarrow L^p(\o)$. The result now follows from Theorems \ref{thm:main_block} and \ref{thm:main_fullynontrivial}.
\end{proof}

Theorem \ref{thm:entire subcritical} is a special case of this result.

\begin{proof}[Proof of Theorem \ref{thm:entire subcritical}]
The first statement follows from Theorem \ref{thm:exterior} with $\o=\rn$ and $G=O(N)$. 

For the second one we take $G$ to be the group generated by $K\cup\{\varrho\}$, where $K:=O(2)\times O(2)\times O(N-2)$ and $\varrho$ is the reflection given by $\varrho(x,y,z):=(y,x,z)$ for $(x,y,z)\in\r^2\times\r^2\times\r^{N-4}$, $N\geq 4$. Note that the $K$-orbit of $\bar x=(x,y,z)$ is $S^1_{|x|}\times S^1_{|y|}\times S^{N-5}_{|z|}$, where $S^{n-1}_r:=\{x\in\r^n:|x|=r\}$. So $K\bar x$ is infinite for every $\bar x\neq 0$ if $N\neq 5$. Let $\phi$ be the homomorphism given by $\phi(g)=1$ if $g\in K$ and $\phi(\varrho)=-1$. The result now follows from Theorem \ref{thm:exterior}.
\end{proof}

The latter symmetries were introduced in \cite{bwi} to prove existence of nonradial solutions to a Schrödinger equation, see also \cite[Theorem 1.37]{w}.

\subsection{Entire solutions of critical systems}

In the critical case linear group actions do not provide compactness. One needs to consider conformal actions.

Let $\Gamma$ be a closed subgroup of $O(N+1)$. Then $\Gamma$ acts isometrically on the unit sphere $\mathbb{S}^N:=\{x\in\r^{N+1}:|x|=1\}$. The stereographic projection $\sigma:\mathbb{S}^N\to\rn\cup\{\infty\}$ induces a conformal action of $\Gamma$ on $\mathbb{R}^N$, given by 
\begin{equation*}
(\gamma,x)\mapsto\tilde{\gamma}x, \qquad\text{where }\;\tilde{\gamma}:=\sigma\circ\gamma^{-1}\circ\sigma^{-1}:\rn\to\rn,
\end{equation*}
which is well defined except at a single point. The group $\Gamma$ acts on the Sobolev space $D^{1,2}(\mathbb{R}^N)$ by linear isometries as follows:
$$\gamma u:=|\det \tilde{\gamma}'|^{1/2^*}u\circ\tilde{\gamma},\qquad\text{for any }\gamma\in\Gamma\;\text{ and }\;u\in D^{1,2}(\mathbb{R}^N);$$
see \cite[Section 3]{cp}. Set
$$D^{1,2}(\rn)^\Gamma:=\{u\in D^{1,2}(\rn):\gamma u=u\text{ for all }\gamma\in\Gamma\}.$$
The argument in \cite[Lemma 3.2]{cp} shows that this space is infinite dimensional if $\Gamma\xi\neq\mathbb{S}^N$ for every $\xi\in\mathbb{S}^N$. If $\phi:\Gamma\to\z_2$ is a continuous homomorphism of groups satisfying $(\phi)$, we define
$$D^{1,2}(\rn)^\phi:=\{u\in D^{1,2}(\rn):\gamma u=\phi(\gamma) u\text{ for all }\gamma\in\Gamma\}.$$
Clearly, Theorems \ref{thm:main_block} and \ref{thm:main_fullynontrivial} hold true for $D^{1,2}(\rn)^\Gamma$ and $D^{1,2}(\rn)^\phi$ as well.

\begin{proof}[Proof of Theorem \ref{thm:entire critical}]
Let $\Gamma=O(m)\times O(n)$ with $m+n=N+1$ and $m,n\geq 2$ act on $\r^{N+1}\equiv\r^m\times\r^n$ in the obvious way. Then $D^{1,2}(\rn)^\Gamma\hookrightarrow L^\frac{2N}{N-2}(\rn)$ is a compact embedding \cite[Proposition 3.3 and Example 3.4]{cp}. So Theorems \ref{thm:main_block} and \ref{thm:main_fullynontrivial} yield a fully nontrivial least energy solution whose components belong to $D^{1,2}(\rn)^\Gamma$. Replacing each component by its absolute value gives a solution whose components are positive.

To prove the second statement we write $\r^{N+1}\equiv\r^2\times\r^2\times\r^{N-3}$ and consider the group $\Gamma$ generated by $K:=O(2)\times O(2)\times O(N-3)$ and the reflection given by $\varrho(x,y,z):=(y,x,z)$ for $(x,y,z)\in\r^2\times\r^2\times\r^{N-3}$, and we take $\phi$ to be the homomorphism defined by $\phi(g)=1$ if $g\in K$ and $\phi(\varrho)=-1$. Then $D^{1,2}(\rn)^K\hookrightarrow L^\frac{2N}{N-2}(\rn)$ is a compact embedding if $N=3$ or $N\geq 5$, and Theorems \ref{thm:main_block} and \ref{thm:main_fullynontrivial} yield a fully nontrivial solution whose components belong to $D^{1,2}(\rn)^\phi$.
\end{proof}

\section{Synchronized solutions of cooperative systems}
\label{sec:locked}

Throughout this section we assume that the system \eqref{eq:system} is purely cooperative. We also assume that $\lambda_i=\lambda$ for all $i$.

Let $u$ be a nontrivial solution to the equation
$$-\Delta u + \lambda u = |u|^{2p-2}u,\qquad u\in H,$$
Then $\bf u=(c_1 u,\dots,c_\ell u)$ is a solution to the system \eqref{eq:system} iff $\bf c=(c_1,\ldots,c_\ell)\in\r^\ell$ solves  the algebraic system \eqref{ss1}.
The solutions to \eqref{ss1} are the critical points of the $\cC^1$-function $J:\r^\ell\to\r$ defined by
$$J(\bf c):=\frac12 |\bf c|^2-\frac{1}{2p}\sum_{i,j=1}^\ell\beta_{ij}|c_j|^p|c_i|^p, \qquad\text{with \ }|\bf c|^2:=\sum_{i=1}^\ell c_i^2.$$
The nontrivial ones belong to the set
\begin{equation*}\label{ss2}
M:=\{\bf{c}\in\r^\ell:\bf c\neq \bf 0,\  \langle\nabla J(\bf c),\bf c\rangle=0\},
\end{equation*}
which is a closed $\cC^1$-submanifold of $\r^\ell$. Hence, there exists $\bf  c\in M$ such that
$$\min_M J=J(\bf c)$$
and $\bf c$ is a solution to the algebraic system \eqref{ss1}. Observe that
\begin{equation*}\label{ss3} 
\langle\nabla J(\bf c),\bf c\rangle=|\bf c|^2- \sum_{i,j=1}^\ell\beta_{ij}|c_j|^p|c_i|^p 
\end{equation*}
and so  
\begin{equation*}\label{ss4}
J(\bf c)=\tfrac{p-1}{p}|\bf c|^2\qquad\text{for any \ } \bf c\in M. 
\end{equation*}

Following the idea we used to prove Theorem \ref{thm:main_fullynontrivial}, we analyze whether all components of a minimizing solution to \eqref{ss1} are nontrivial.

\begin{lemma} \label{lem:synchronized}
Let $\bf c=(c_1,\ldots,c_\ell)\in M$ be such that $J(\bf c)=\min_M J$.
\begin{itemize}
\item[$(i)$] If $p<2$, then $c_i\neq 0$ for every $i=1,\ldots,\ell$.
\item[$(ii)$] If $p=2$, then $\sum_{j\neq i}\beta_{ij}c_j^2\leq 1$ for every $i=1,\ldots,\ell$. 
\end{itemize}
\end{lemma}

\begin{proof}
$(i):$ Let $p<2$ and assume by contradiction that one component of $\bf c$ is zero, say $c_1=0$. Set
$\bf c_\eps:=(\eps,c_2,\dots,c_\ell)$. As $\bf c\in M$, we have
$$|\bf c|^2=\sum_{i\ge2}c_i^2= \sum_{i,j\ge2}\beta_{ij}|c_j|^p|c_i|^p.$$
Hence, there exists $\eps_0>0$ such that for every $\eps\in (-\eps_0,\eps_0)$ there is a unique $t=t(\eps)>0$ solving 
\begin{align*}
t^2\Big(\eps^2+\sum_{i\ge2}c_i^2\Big)&=t^{2p}\Big(\beta_{11}\eps^{2p}+2\eps^p\sum_{j\ge2}\beta_{1j}|c_j|^p +\sum_{i,j\ge2}\beta_{ij}|c_j|^p|c_i|^p\Big) \\
&= t^{2p}\Big(\beta_{11}\eps^{2p}+2\eps^p\sum_{j\ge2}\beta_{1j}|c_j|^p +\sum_{i\ge2}c_i^2\Big),
\end{align*}
namely,
$$t(\eps)=\left(\frac{\eps^2+|\bf c|^2}{\beta_{11}\eps^{2p}+2\eps^p\sum_{j\ge2}\beta_{1j}|c_j|^p +|\bf c|^2}\right)^\frac{1}{2p-2}.$$
Moreover, $\eps\mapsto t(\eps)$ is a $\cC^1$-function, $ t(0)=1$ and $t'(0)= 0$.

So, if $p<2$, then
$$t(\eps)=1-\frac{1}{p-1}\eps^p\frac{1}{|\bf c|^2}\sum_{j\ge2}\beta_{1j}|c_j|^p+o(\eps^p).$$
Therefore,
\begin{align*}
J(\bf c)-J(\bf c_\epsilon)&=\frac{p-1}{p}\Big[\sum_{i\ge2} c_i^2-t^2\Big(\eps^2+\sum_{i\ge2} c_i^2\Big)\Big] \\
&=\frac{p-1}{p}\Big[(1-t^2)|\bf c|^2-t^2\eps^2\Big] \\
&=\frac{p-1}{p}\Big(\frac4{2p-2}\eps^p\sum_{j\ge2}\beta_{1j}|c_j|^p+o(\epsilon^p)\Big)\\
&=\frac{2}{p}\eps^p\Big(\sum_{j\ge2}\beta_{1j}|c_j|^p+o(1)\Big)>0,
\end{align*}
which is a contradiction.

$(ii):$ Let $p=2$ and assume by contradiction that $\sum_{j\neq i}\beta_{ij}c_j^2>1$ for some $i$. Then, as $\bf c$ satisfies \eqref{ss1}, we have that $c_i=0$. Let us assume, for simplicity, that $i=1$. Then, defining $\bf c_\eps$ and $t(\eps)$ as before we have that
$$t(\eps)=1+\frac12\epsilon^2\frac1{|\bf c|^2}\Big(1-2\sum_{j\ge2}\beta_{1j}|c_j|^{2}\Big)+o(\eps^2)$$
and, therefore,
\begin{align*}
J(\bf c)-J(\bf c_\eps)&=\frac{1}{2}\Big[(1-t^2)\sum_{i\ge2}c_i^2-t^2\eps^2\Big]\\
&=\frac12\Big(\eps^2(2\sum_{j\ge2}\beta_{1j}c_j^2-2)+o(\epsilon^2)\Big)\\
&=\eps^{2}\Big(\sum_{j\ge2}\beta_{1j} c_j^2-1 +o(1)\Big)>0.
\end{align*}
This is a contradiction.
\end{proof}

These computations highlight the different behavior of the system for $p<2$ and $p=2$. To complete the picture observe that, if $p>2$, then
$$t(\eps)=1+\frac1{2p-2}\eps^2\frac1{|\bf c|^2} +o(\epsilon^{ 2}).$$
Therefore
\begin{align*}
J(\bf c)-J(\bf c_\eps)&=\frac{p-1}{p}\Big[(1-t^2)|\bf c|^2-t^2\eps^2\Big]\\
&=\frac{p-1}{p}\left[ -\eps^2+o(\eps^2)\right].
\end{align*}
We conjecture that for $p>2$ there is no fully synchronized solution. 
\smallskip

\begin{proof}[Proof of Theorem \ref{thm:synchronized}]
This result follows immediately from Lemma \ref{lem:synchronized}$(i)$.
\end{proof}

\bigskip

\begin{flushleft}
\textbf{Mónica Clapp}\\
Instituto de Matemáticas\\
Universidad Nacional Autónoma de México\\
Circuito Exterior, Ciudad Universitaria\\
04510 Coyoacán, Ciudad de México, Mexico\\
\texttt{monica.clapp@im.unam.mx} 
\medskip

\textbf{Angela Pistoia}\\
Dipartimento di Metodi e Modelli Matematici\\
La Sapienza Università di Roma\\
Via Antonio Scarpa 16 \\
00161 Roma, Italy\\
\texttt{angela.pistoia@uniroma1.it} 
\end{flushleft}

\end{document}